\documentclass[12pt]{amsart}
\usepackage{amssymb, enumerate, mdwlist}

\evensidemargin\paperwidth
\advance\evensidemargin-\textwidth
\advance\oddsidemargin-1in
\evensidemargin\oddsidemargin

\topmargin\paperheight
\advance\topmargin-\textheight
\advance\topmargin-1in

\theoremstyle{plain}
\newtheorem{theorem}{Theorem}[section]
\theoremstyle{plain}
\newtheorem{proposition}[theorem]{Proposition}
\theoremstyle{plain}
\newtheorem{lemma}[theorem]{Lemma}
\theoremstyle{plain}

\theoremstyle{plain}

\theoremstyle{plain}

\theoremstyle{definition}
\newtheorem{definition}[theorem]{Definition}
\theoremstyle{remark}
\newtheorem{remark}[theorem]{Remark}
\theoremstyle{remark}

\theoremstyle{remark}

\title[Superrigidity via quasi-cocycles]
{Superrigidity from Chevalley groups into acylindrically hyperbolic groups via quasi-cocycles.}
\author{Masato Mimura}  \thanks{The author was supported in part by KAKENHI 25800033, the Grant-in-Aid for Young Scientists (B)}
\address{Masato Mimura\\
Mathematical Institute, Tohoku University}
\email{mimura-mas@m.tohoku.ac.jp}
\date{\today}

\begin{document}

\begin{abstract}
We prove that every homomorphism from the elementary Chevalley group over a finitely generated unital commutative ring associated with a reduced irreducible classical root system of rank at least $2$, and ME analogues of such groups, into acylindrically hyperbolic groups has an absolutely elliptic image. This result provides a non-arithmetic generalization of homomorphism superrigidity of Farb--Kaimanovich--Masur and Bridson--Wade.
\end{abstract}

\keywords{Property (T); Quasi-cocycles; Elementary Chevalley groups; Acylindrically hyperbolic groups}

\maketitle

\section{Main result}
The celebrated \textit{Farb--Kaimanovich--Masur superrigidity theorem} \cite{KaimanovichMasur}, \cite{FarbMasur} states that every homomorphism from an (irreducible) higher rank lattice into $\mathrm{MCG}(\Sigma_g)$, the mapping class group of a closed oriented surface $\Sigma_g$ of genus $g$, has \textit{finite} image. Later, Bridson and Wade \cite{BridsonWade} showed that the same superrigidity remains true if the target group is replaced with $\mathrm{Out}(F_N)$, the outer automorphism group of a (non-abelian) free group $F_N$ of finite rank $N$. In an unpublished manuscript of \cite{Mimuraup}, the present author obtained a similar homomorphism superrigidity from \textit{$($commutative$)$ universal lattices} and \textit{symplectic universal lattices}, that means, groups of the form $\mathrm{SL}(n,\mathbb{Z}[x_1,\ldots ,x_k])$ with $n\geq 3$, and $\mathrm{Sp}(2n,\mathbb{Z}[x_1,\ldots ,x_k])$ with $n\geq 2$, where $k$ finite. 

In this paper, we present a full generalization of this homomorphism superrigidity, as follows.

\begin{theorem}[Main Theorem]\label{theorem=main}
Let $\Phi$ be a reduced irreducible classical root system of rank at least $2$. Let $A$ be a finitely generated, unital, commutative, and associative ring. Let $\Gamma$ and $\Lambda$ be as follows:
\begin{itemize}
  \item The group $\Gamma$ is a quotient of a finite index subgroup of the $($simply connected$)$ elementary Chevalley group $\mathrm{E}(\Phi,A)$.
  \item The group $\Lambda$ is measure equivalent to $\Gamma$ with the $L_2$-integrability condition on a corresponding $\mathrm{ME}$ cocycle from $\Gamma$ to $\Lambda$. $($See the proof of Proposition~$\ref{proposition=ME}$ for the definition of the $L_2$-integrability condition.$)$
\end{itemize}

Then the following hold true.
\begin{enumerate}[$(i)$]
 \item For every acylindrically hyperbolic group $G$, every group homomorphism from $\Lambda$ into $G$ has an $\mathrm{absolutely}$ $\mathrm{elliptic}$ image $H$. That means, every acylindrical $G$-action by isometries on a $($Gromov-$)$hyperbolic geodesic space has a bounded $H$-orbit.
 \item Furthermore, if $G$ is of the form $\mathrm{MCG}(\Sigma_g)$ or $\mathrm{Out}(F_N)$ for $g$, $N$ finite, then every group homomorphism from $\Lambda$ into $G$ has a finite image.
\end{enumerate}
\end{theorem}

\begin{remark}\label{remark=notfg}
In fact, in the assertions for $\Lambda=\Gamma$ in Theorem~\ref{theorem=main}, we can drop the assumption of the finite generation of the ring $A$. See Section~\ref{section=notfg}.
\end{remark}

We explain some terminology in the theorem above, firstly on $\Gamma$ and $\Lambda$. For (simply connected) \textit{elementary Chevalley groups}, we refer to \cite{Steinberg} and \cite{ErshovJaikinKassabov}. A basic example is the \textrm{elementary group} $\mathrm{E}(n,A)$, when $\Phi=A_{n-1}$, which is a subgroup of $\mathrm{GL}(n,A)$ generated by the elementary matrices $e_{i,j}^a$ for $1\leq i\leq n$, $1\leq j\leq n$, $i\ne j$, and $a\in A$. Here, $e_{i,j}^a$ denotes the following matrix in $\mathrm{GL}(n,A)$: $1$ on diagonal, $0$ on all but the $(i,j)$-th off-diagonal entries, and $a$ on the $(i,j)$-the entry. Observe the following commutator relation: $[e_{i,j}^{a_1}, e_{j,k}^{a_2}]=e_{i,k}^{a_1a_2}$ for $i\ne j\ne k\ne i$, where $[g,h]:=ghg^{-1}h^{-1}$. It implies that if $n\geq 3$ and if $A$ is as in Theorem~\ref{theorem=main}, then $\mathrm{E}(n,A)$ is a finitely generated group. If we let $\Phi=C_n$, then $\mathrm{E}(\Phi,A)$ is the elementary symplectic group. By setting $A=\mathbb{Z}[x_1,\ldots, x_k]$ in both cases above, we recover the case of (commutative) universal lattices and of symplectic universal lattices. (More precisely, here we use results of Suslin and Kope\u{\i}ko \cite{Kopeiko})

For the \textit{measure equivalence} (ME) and the $L_p$\textit{-integrability condition}, the reader may consult with \cite{Furmansurvey}, see Definition~2.1, Subsection~2.3.2, and Appendix~A.3 therein; we will also briefly recall it in the proof of Proposition~\ref{proposition=ME} in the present paper. A basic example of a measure equivalent pair $(\Gamma, \Lambda)$ is a pair of two lattices in a common locally compact second countable group. Moreover, if these two in the example above are irreducible higher rank lattices, then this measure equivalence is known to satisfy the $L_p$-integrability condition for all $p\in [1,\infty)$, see for instance \cite[Section~8]{BaderFurmanGelanderMonod}. Therefore, item $(ii)$ of Theorem~\ref{theorem=main} generalizes the Farb--Kaimanovich--Masur and the Bridson--Wade superrigidity of higher rank lattices, provided that the corresponding higher rank algebraic group has no rank one factor, to groups possibly with no arithmetic backgrounds.

Secondly, on $G$, recall from Dahmani--Guirardel--Osin \cite{DahmaniGuirardelOsin} and Osin \cite{Osin} that a group $G$ is said to be \textit{acylindrically hyperbolic} if $G$ admits a \textit{non-elementary} \textit{acylindrical} action by isometries on a (Gromov-)hyperbolic geodesic space. Here an isometric $G$-action on a metric space $S$ is said to be \textit{acylindrical} if for every $\epsilon >0$, there exist $R,N>0$ such that for every two points $x,y \in S$ with $d(x,y)\geq R$, there are at most $N$ elements $g\in G$ that satisfy $d(x,gx)\leq \epsilon$ and $d(y,gy)\leq \epsilon$. Recall also that an acylindrical $G$-action by isometries on a hyperbolic geodesic space $S$ is said to be \textit{non-elementary} if the limit set of $G$ on the Gromov boundary $\partial S$ contains more than $2$ points. The class of acylindrical hyperbolic groups contains $\mathrm{MCG}(\Sigma_g)$, $\mathrm{Out}(F_N)$ for $g\geq 1$ and $N\geq 2$, and  infinitely presented graphical $C(7)$-small cancellation groups \cite{GruberSisto}, which include Osajda's monsters \cite{Osajda}, \cite{Osajda2017} (for more examples, see \cite[Appendix~A]{Osin}). Note that, because every non-elementary free product is acylindrically hyperbolic, there is \textit{no} hope to expect that the image is finite in $(i)$ of Theorem~\ref{theorem=main} in general. Item $(ii)$ implies that, if we know all of the absolutely elliptic subgroups in $G$ to certain extent, then it may be possible to deduce that the image is, in fact, finite.

\begin{remark}\label{remark=Zaverse}
In \cite{BridsonWade}, Bridson and Wade defined that a group is  $\mathbb{Z}$\textit{-averse} if no finite index subgroup admits a normal subgroup that surjects onto $\mathbb{Z}$. Then, they showed that for every $\mathbb{Z}$-averse group, the homomorphism superrigidity into $\mathrm{MCG}(\Sigma_g)$ and into $\mathrm{Out}(F_N)$ holds true. Note that higher rank lattices are $\mathbb{Z}$-averse by the Margulis normal subgroup theorem together with the Matsushima vanishing theorem. We remark, on the other hand,  that many examples of $\Lambda$ in Theorem~\ref{theorem=main} are \textit{not} $\mathbb{Z}$-averse. Indeed, for $\Lambda=\Gamma =\mathrm{SL}(3,\mathbb{Z}[x])$, the kernel $K$ of the substitution map with $x=0$ surjects onto $\mathbb{Z}$. To see this, observe that the derivation cocycle $\Lambda \to (\mathrm{Mat}(3,\mathbb{Z}),+)$; $g \mapsto g'\mid_{x=0}$ is a group homomorphism if it is restricted on $K$.
\end{remark}

\begin{remark}\label{remark=hyperbolicemb}
Acylindrical hyperbolicity can be also characterized in terms of \textit{hyperbolically embedded subgroups}, see \cite[Theorem~1.2]{Osin}. The reader familiar with relative hyperbolic groups might think that the conclusion in item $(i)$  of Theorem~\ref{theorem=main} may be stated in terms of such subgroups (not in terms of the absolute ellipticity). However, Theorem~7.7 in \cite{MinasyanOsin} implies that it is impossible in general.
\end{remark}

\ 

\noindent
\textit{Convention.} Unless otherwise stated, groups are assumed to be countable and discrete.

\section{Strategy of the proof of Theorem~\ref{theorem=main}}\label{section=sketch}

To prove the main theorem, we introduce the following properties, which are motivated by property $(\mathrm{TT})$ of Monod \cite{Monod}.

\begin{definition}\label{definition=TTwm}
\begin{enumerate}[$(1)$]
  \item A group $\Gamma$ is said to have \textit{property $(\mathrm{TT})/\mathrm{T}$} (\textit{``property $\mathrm{TT}$ modulo $\mathrm{T}$"}) if all quasi-cocycles into every unitary $\Gamma$-representation that does not contain trivial representation are bounded.
  \item A group $\Gamma$ is said to have \textit{property $(\mathrm{TT})_{\mathrm{wm}}$} (\textit{``property $\mathrm{TT}$ for weakly mixing representations"}) if all quasi-cocycles into every unitary $\Gamma$-representation that is weakly mixing are bounded.
\end{enumerate}
\end{definition}

We refer the reader to  Section~\ref{section=TTwm} for the precise definitions of \textit{quasi-$(1$-$)$cocycles} and of \textit{weakly mixing} unitary representations. We also see in  Section~\ref{section=TTwm} that property $(\mathrm{TT})_{\mathrm{wm}}$ is strictly stronger than Kazhdan's property $(\mathrm{T})$, and strictly weaker than property $(\mathrm{TT})$.

The proof of Theorem~\ref{theorem=main} consists of the following two parts: 

\begin{theorem}\label{theorem=TTwm}
Let $\Lambda$ be as in Theorem~$\ref{theorem=main}$. Then $\Lambda$ has property $(\mathrm{TT})_{\mathrm{wm}}$.
\end{theorem}

\begin{proposition}\label{proposition=HullOsin}
Let $\Lambda$ be a group. If $\Lambda$ has property $(\mathrm{TT})_{\mathrm{wm}}$, then every homomorphism from $\Lambda$ into an acylindrically hyperbolic group has an absolutely elliptic image.
\end{proposition}

We will prove Theorem~\ref{theorem=TTwm} in Section~\ref{section=TTwmforE}, and Proposition~\ref{proposition=HullOsin} and Theorem~\ref{theorem=main} in Section~\ref{section=ah}. For the proof of Theorem~\ref{theorem=TTwm}, more precisely, we firstly prove Theorem~\ref{theorem=TT/T}, which states that the $\Gamma$ (elementary Chevelley groups in our settings) possesses property $(\mathrm{TT})/\mathrm{T}$, and then deduce property $(\mathrm{TT})_{\mathrm{wm}}$ for $\Lambda$ by the $L_2$-induction process (item $(2)$ in Proposition~\ref{proposition=ME}). 

The proof of Theorem~\ref{theorem=TT/T} is the main part of this paper. For the proof, we employ a \textit{new} observation on \textit{bounded generation} for elementary Chevalley groups over commutative rings, see Lemma~\ref{lemma=boundedgeneration} for the precise statement. Here we state the definition of bounded generation in our convention. 

\begin{definition}\label{definition=bg}
For two non-empty subsets $e_{\Gamma}\in X \subseteq \Gamma$, $ Y \subseteq \Gamma$ of a group $\Gamma$, we say that $Y$ is \textit{boundedly generated} by $X$ if there exists a positive  integer $N$ such that $Y\subseteq X^N$ holds true. Here, $X^N(\subseteq G)$ denotes the set of all the products of possibly overlapping $N$ elements in $X$.
\end{definition}
Note that in some literature, the terminology of bounded generation is restricted to the case where $X$ is a finite union of cyclic subgroups and where $Y=\Gamma$. Our convention is much more general.

The use of bounded generation has been revealed to be a powerful tool in the proof of property $(\mathrm{T})$ and fixed point properties; see for instance \cite{Shalom1999} and \cite{Shalom2006}. In general, bounded generation is extremely hard to check, and there are only few known useful examples. As we mentioned above, our new bounded generation, Lemma~\ref{lemma=boundedgeneration}, is weaker than all other nontrivial examples which are known. However, in return for this weakness, it applies to considerably wide classes of Chevalley groups. We expect that Lemma~\ref{lemma=boundedgeneration} has some potential for further applications.

We remark on some relevant results by other researchers. Another generalization of use of quasi-cocycles to homomorphism superrigidity rather than \textit{quasi-homomorphisms}, which was observed firstly in \cite{Mimuraup}, is obtained by Burger and Iozzi \cite{BurgerIozzi} (note that the use of quasi-homomorphisms to homomorphism superrigidity was initiated by \cite{BestvinaFujiwara}, but that the use of (twisted) quasi-cocycles had  not been observed before \cite{Mimuraup}; see also Remark~\ref{remark=quasi-hom}). They introduce the $\ell_2$\textit{-stability}, and prove certain type of homomorphism superrigidity of $\ell_2$-stable groups into $\mathrm{MCG}(\Sigma_g)$. Some discrete infinite groups coming from  products of trees have the $\ell_2$-stability, and they never have property $(\mathrm{T})$ (thus, not $(\mathrm{TT})_{\mathrm{wm}}$ either). Hence, their result applies to more groups (the $\ell_2$-stability is implied by property $(\mathrm{TT})_{\mathrm{wm}}$; see Proposition~\ref{proposition=ME}). The $\ell_2$-stability, on the other hand, is not stable under the measure equivalence: for instance, every uniform and irreducible lattice in $\mathrm{SL}(2,\mathbb{R})\times \mathrm{SL}(2,\mathbb{R})$ enjoys the $\ell_2$-stability, whereas $\pi_1(\Sigma_g)\times \pi_1(\Sigma_g)$ does not for all $g\geq 2$. These two groups are ($L_{\infty}$-integrable) ME.

\section{Property $(\mathrm{TT})/\mathrm{T}$ and Property $(\mathrm{TT})_{\mathrm{wm}}$ }\label{section=TTwm}

As we mentioned in Section~\ref{section=sketch}, we employ \textit{property $(T)$-like properties} for the proof of Theorem~\ref{theorem=main}. We refer the reader to \cite{BekkadelaHarpeValette} for a comprehensive treatise on Kazhdan's property $(\mathrm{T})$.

Let $\Gamma$ be a (countable discrete) group, and $(\pi, \mathcal{H})$ be a unitary $\Gamma$-representation. A map $c\colon \Gamma \to \mathcal{H}$ is called a \textit{quasi-}($1$-)\textit{cocycle} into $\pi$ if 
\[
\sup_{g,h\in \Gamma} \|c(gh)-c(g)-\pi(g)c(h)\|<\infty. 
\]
The quantity in the left-hand side is called the \textit{defect} of $c$. Quasi-($1$-)cocycles are related to (second) \textit{bounded cohomology}, see \cite{Monod} and \cite{MonodICM} for details. Recall that $\Gamma$ has \textit{Kazhdan's property $(\mathrm{T})$} if and only if all genuine cocycles (namely, all quasi-cocycles with zero defect) into every unitary representation $\pi$ are bounded (this is the Delorme--Guichardet theorem; see \cite[Theorem~2.12.4 and Proposition~2.2.9]{BekkadelaHarpeValette}), and that $\Gamma$ is said to have \textit{Monod's property $(\mathrm{TT})$} if all quasi-cocycles into every unitary representation $\pi$ are bounded. Recall in addition that a unitary representation $\pi$ is said to be \textit{weakly mixing} if it does not contain (non-zero) finite dimensional subrepresentations.

By Definition~\ref{definition=TTwm}, the following implications are clear:
\[
(\mathrm{TT}) \quad \Rightarrow \quad (\mathrm{TT})/\mathrm{T} \quad \Rightarrow \quad (\mathrm{TT})_\mathrm{wm} .
\]
\begin{lemma}\label{lemma=TTwmT}
Property $(\mathrm{TT})_{\mathrm{wm}}$ implies property $(\mathrm{T})$. 
\end{lemma}

\begin{proof}
First observe that property $(\mathrm{TT})_{\mathrm{wm}}$ implies that all (genuine) cocycles into every weakly mixing unitary representation are bounded. Suppose that $\Gamma$ has property $(\mathrm{TT})_{\mathrm{wm}}$ but that fails to have property $(\mathrm{T})$. Then by \cite[Proposition~2.12.2.$(ii)$]{BekkadelaHarpeValette}, the failure of property $(\mathrm{T})$ implies that some \textit{weakly mixing} unitary representation $(\pi,\mathcal{H})$ admits almost invariant unit vectors. Here we say that $(\pi,\mathcal{H})$ admits \textit{almost invariant} unit \textit{vectors} if for every $\epsilon>0$ and for every finite subset $K\subseteq \Gamma$, there exists a unit vector $\xi \in \mathcal{H}$ such that $\sup_{k\in K} \|\pi (k)\xi-\xi\|< \epsilon$. From this, we can construct an unbounded cocycle into $(\ell_2$-$\bigoplus \pi ,\ell_2$-$\bigoplus \mathcal{H} )$, where ``$\ell_2$-$\bigoplus$" denotes the countable $\ell_2$-direct sum. We can do that by taking the $\ell_2$-direct sum of countable copies of $(\pi,\mathcal{H})$ and by piling up cocycles of the form $g \mapsto \pi(g)\zeta_n -\zeta_n$ for a suitable sequence of vectors $\{\zeta_n\}_n$. This argument is known as  a \textit{Guichardet-type argument}; see \cite[Proposition~2.12.2.$(ii)$]{BekkadelaHarpeValette} for more details. This contradicts the assumption that $\Gamma$ has property $(\mathrm{TT})_{\mathrm{wm}}$ because the countable $\ell_2$-direct sum of weakly mixing unitary representations is also weakly mixing. 
\end{proof}

Therefore, we have obtained the following implications:
\[
(\mathrm{TT}) \quad \Rightarrow \quad (\mathrm{TT})/\mathrm{T} \quad \Rightarrow \quad (\mathrm{TT})_\mathrm{wm}  \quad \Rightarrow \quad (\mathrm{T}) .\tag{$\star$}
\]

In what follows, we briefly discuss the reverse implications in $(\star)$. The right implication cannot be reversed, because all infinite hyperbolic groups fail to have property $(\mathrm{TT})_\mathrm{wm}$ (see Section~\ref{section=ah}). Burger and Iozzi \cite{BurgerIozzi} constructed a counterexample to the converse implication to the left one. It might be open whether the middle implication can be reversed.

The next proposition explains why we need both property $(\mathrm{TT})/\mathrm{T}$ and property $(\mathrm{TT})_{\mathrm{wm}}$ to prove Theorem~\ref{theorem=main}.
\begin{proposition}\label{proposition=ME}
\begin{enumerate}[$(1)$]
 \item Property $(\mathrm{TT})/\mathrm{T}$ and property $(\mathrm{TT})_{\mathrm{wm}}$, respectively, pass to group quotients and to finite index subgroups.
 \item If $\Gamma$ has property $(\mathrm{TT})/\mathrm{T}$, and if $\Lambda$ is measure equivalent to $\Gamma$ with the $L_2$-integrability condition on an ME cocycle from $\Gamma$ to $\Lambda$, then $\Lambda$ has property $(\mathrm{TT})_{\mathrm{wm}}$.
\end{enumerate}
\end{proposition}

\begin{proof}
Item $(1)$ is straightforward from the pull-back and induction of quasi-cocycles (for the latter part, compare with the argument below). Here, observe that the weak mixing property for unitary representations is preserved by pull-back. In outline, item $(2)$ will be deduced from the $L_2$-induction of quasi-cocycles for the ME coupling $(\Lambda, \Gamma)$. In what follows, we will discuss this in more details. 

Let $(\Omega,m)$ be an ME-coupling for $(\Gamma,\Lambda)$, $D\cong \Omega/\Lambda$ be a Borel $\Lambda$-fundamental domain. Let $\alpha\colon \Gamma \times D \to \Lambda$ be the corresponding ME-cocycle, namely, 
\[
\textrm{for every $\gamma\in \Gamma$ and for almost every $x\in D$,}\quad \gamma \cdot x =\alpha(\gamma, x) \gamma x.
\]
Here $x\stackrel{\gamma \cdot}{\mapsto} \gamma\cdot x$ denotes the natural action $\Gamma \curvearrowright D$, and $x\stackrel{\gamma}{\mapsto} \gamma x$ denotes the original action of $\Gamma$ on $\Omega$. The reader may consult \cite{Furmansurvey} for more details on these notions and definitions. We recall that $\alpha$ is said to satisfy the \textit{$L_2$-integrability condition} if 
\[
\mathrm{for\ every\ \gamma \in \Gamma ,}\quad |\alpha(\gamma,\cdot )|_{\Lambda} \in L_2(D).
\]
Here $|\cdot|_{\Lambda}$ is the length function on $\Lambda$ with respect to some finite generating set. (Here we refer to the following facts. Property $(\mathrm{T})$ for a discrete group implies the finite generation \cite[Theorem~1.3.1]{BekkadelaHarpeValette}. Property $(\mathrm{TT})/\mathrm{T}$ implies property $(\mathrm{T})$ for a countable discrete group by $(\star)$. Property $(\mathrm{T})$ is an ME-invariant \cite[Subsection~3.1.1]{Furmansurvey}.)

Let $(\sigma, \mathfrak{H})$ be a unitary $\Lambda$-representation. First, we define the induced unitary $\Gamma$-representation $(\Omega \mathbf{Ind}_{\Lambda}^{\Gamma}\sigma, \Omega_{\Lambda}^{\Gamma}\mathfrak{H})$ by 
$\Omega_{\Lambda}^{\Gamma}\mathfrak{H}\cong L_2(D,\mathfrak{H})$ with the twisted $\Gamma$-action defined almost everywhere by 
\[
((\Omega \mathbf{Ind}_{\Lambda}^{\Gamma}\sigma)(\gamma) f)(x):=
\sigma (\alpha(\gamma^{-1},x)^{-1})f(\gamma^{-1}\cdot x).
\]
Then $\Omega \mathbf{Ind}_{\Lambda}^{\Gamma}\sigma$ becomes a unitary $\Gamma$-representation. Secondly, let $c\colon \Lambda \to \mathfrak{H}$ be a quasi-$\sigma$-cocycle. Then, we hope to define the induced map $\tilde{c}$ by the following formula:
\[
\textrm{For $\gamma\in \Gamma$ and $x\in D$, }\quad \tilde{c}(\gamma)(x):=c(\alpha(\gamma^{-1},x)^{-1}).
\]
If it is well-defined, then this gives rise to a quasi-cocycle into $\Omega \mathbf{Ind}_{\Lambda}^{\Gamma}\sigma$.

We warn that there is, in general, an issue of the $2$-summability to obtain the well-definedness of $\tilde{c}$. However, in our case, the assumption of the $L_2$-integrability of $\alpha$ ensures this summability. We, furthermore, warn that even if $\sigma \not \supseteq 1_{\Lambda}$, it may happen that $\Omega \mathbf{Ind}_{\Lambda}^{\Gamma}\sigma   \supseteq 1_{\Gamma}$. The argument in \cite[Subsection~4.1.1]{Furmansurvey}, however, shows the following: if $\sigma$ is, besides, weakly mixing, then we conclude that $\Omega \mathbf{Ind}_{\Lambda}^{\Gamma}\sigma  \not \supseteq 1_{\Gamma}$ (this is the role of the weak mixing property in the current paper). 

By observing these two subtle points above, we start from an arbitrarily taken \textit{weakly mixing} unitary representation $\sigma$ of $\Lambda$ and an arbitrarily taken quasi-cocycle $c$ into $\sigma$. Then, by the $L_2$-induction process above, we obtain the unitary representation $\Omega \mathbf{Ind}_{\Lambda}^{\Gamma}\sigma $, \textit{that does not contain trivial representation}, and the quasi-cocycle $\tilde{c}$ into $\Omega \mathbf{Ind}_{\Lambda}^{\Gamma}\sigma $. Then, by property $(\mathrm{TT})/\mathrm{T}$ for $\Gamma$, we have that $\tilde{c}$ is bounded.

Initially, there is a gap in deducing the boundedness of $c$ from that of $\tilde{c}$. However, a deep result in \cite[Theorem 4.4]{MonodShalom}, which concerns inductions of bounded cohomology, enables us to take this deduction without gaps in our (unitary) setting. Therefore, $c$ must be bounded and $\Lambda$ enjoys property $(\mathrm{TT})_{\mathrm{wm}}$, as desired.
\end{proof}

The next lemma is the key observation to Section~\ref{section=TTwmforE}, and it is a development of \cite[Proposition~6.6]{Mimura1}. Recall from Definition~\ref{definition=bg} on bounded generations.

\begin{definition}\label{definition=relTT}
A pair $\Gamma\supseteq Z$ of a group and a subset is said to have \textit{relative property} $(\mathrm{TT})$ if all quasi-cocycles on $\Gamma$ into every unitary $\Gamma$-representation are bounded on $Z$. 
\end{definition}

\begin{lemma}\label{lemma=Shalom}
Let $\Gamma$ be a $($countable discrete$)$ group, $\Gamma_0$ be a subgroup of $\Gamma$, and $Z\ni e_{\Gamma}$ be a $\mathrm{subset}$ of $\Gamma$. Assume that the triple $(\Gamma,\Gamma_0,Z)$ satisfies the following three conditions.
\begin{enumerate}[$(i)$]
  \item The subset $Z$ generates $\Gamma$ $($as a group$)$.
  \item For every $h\in \Gamma_0$, $h^{-1}Zh \subseteq Z$.
  \item The group $\Gamma$ is boundedly generated by $\Gamma_0\cup Z$.
\end{enumerate}
Assume, besides,  that the pair $\Gamma \supseteq Z$ has relative property $(\mathrm{TT})$ and that $\Gamma$ has property $(\mathrm{T})$. Then, $\Gamma$ has property $(\mathrm{TT})/\mathrm{T}$.

In particular, if $(\Gamma, \Gamma_0=\Gamma, Z)$ fulfills conditions $(i)$ and $(ii)$, then relative property $(\mathrm{TT})$ for the pair $\Gamma \supseteq Z$ together with property $(\mathrm{T})$ for $\Gamma$ implies property $(\mathrm{TT})/\mathrm{T}$ for $\Gamma$.
\end{lemma}
Note that $Z\subseteq \Gamma$ is in general \textit{not} a subgroup of $\Gamma$: compare with condition $(i)$.

For the proof, recall the original definition of Kazhdan's property $(\mathrm{T})$: $\Gamma$ is said to have property $(\mathrm{T})$ if every unitary representation that does not contain trivial representation admits almost invariant unit vectors.  See the proof of Lemma~\ref{lemma=TTwmT} for the definition of existence of almost invariant vectors. 

\begin{proof}
Let $(\pi,\mathcal{H})$ be a unitary $\Gamma$-representation that does not contain trivial representation, and $c\colon \Gamma\to \mathcal{H}$ be a quasi-cocycle into $\pi$. Let $C$ be the maximum of the defect of $c$ and $\sup_{z\in Z} \|c(z)\|$. By relative property $(\mathrm{TT})$, this $C$ is finite.

We claim that $c(\Gamma_0)$ is bounded. Indeed, for every $h\in \Gamma_0$ and $z\in Z$, we have that by $(ii)$,
\begin{align*}
    &\|\pi(z)c(h) - c(h)\| \leq \|c(zh) - c(h)-c(z)\|+C \\
\leq  & \|c(h (h^{-1}zh)) - c(h)\|+2C \leq \|c(h)-c(h)\| +4C =4C.
\end{align*}
Suppose, on the contrary, that $c(\Gamma_0)$ is unbounded. Take $\{h_n\}_n\subseteq  \Gamma_0$ such that $\|c(h_n)\| \to  \infty$. Then, we deduce that $\{c(h_n)/\|c(h_n)\|\}_n$ forms a sequence of almost invariant unit vectors for $\pi$. To see this, observe that for every finite subset $K$ of $G$, there exists $l\in \mathbb{N}$ such that $(Z\cup Z^{-1})^l\supseteq K$ by $(i)$. However, this is absurd because $\Gamma$ has property $(\mathrm{T})$. Hence, $c(\Gamma_0)$ is bounded. Finally, we deduce the boundedness of $c(\Gamma)$ with the aid of bounded generation (condition $(iii)$) and a triangle inequality of $c$ up to $+C$ error.

For the last part of the assertions in Lemma~\ref{lemma=Shalom}, observe that $\Gamma$ is boundedly generated by $\Gamma$ itself.
\end{proof}

This lemma is inspired by \cite[Section~4.III]{Shalom2006}. The first application of Lemma~\ref{lemma=Shalom} was in \cite[Proposition~6.6]{Mimura1}, where we use for the tuple $(\Gamma,\Gamma_0,Z)=(\mathrm{E}(n,A),\mathrm{E}(n,A)\cap \mathrm{GL}(n-1,A), M \cup L)$ for $A$ as in Theorem~$\ref{theorem=main}$. Here $\mathrm{GL}(n-1,A)$ sits on the left-upper corner, $M:=\langle e_{i,n}^a : 1\leq i\leq n-1 , a\in A\rangle(\simeq (A^{n-1},+))$, and  $L:=\langle e_{n,j}^a : 1\leq j\leq n-1 , a\in A\rangle (\simeq (A^{n-1},+))$. We warn that, in this case, condition $(iii)$ (bounded generation) is proved by Vaserstein \cite[Corollary~3]{Vaserstein}, but that for other root systems $\Phi$, it might be open whether similar bounded generation remains true. As we will see in Section~\ref{section=TTwmforE}, in the present paper,  we will employ much weaker bounded generation (Lemma~\ref{lemma=boundedgeneration}). Owing to it, we are able to apply Lemma~\ref{lemma=Shalom} to the case of that $\Gamma_0=\Gamma$, where condition $(iii)$ is trivially fulfilled.

\section{Key theorem: property $(\mathrm{TT})/\mathrm{T}$ for elementary Chevalley groups}\label{section=TTwmforE}

As we mentioned in Section~\ref{section=sketch}, the following theorem is the essential part in the present paper, and may be regarded as a strengthening of \cite[Theorem~1.1]{ErshovJaikinKassabov} for elementary Chevalley groups (recall Lemma~\ref{lemma=TTwmT}). Note that, however, because we employ property $(\mathrm{T})$ for such groups in the proof below, our theorem is based on their result.

\begin{theorem}\label{theorem=TT/T}
Let $\Phi$ and $A$ be as in Theorem~$\ref{theorem=main}$. Then $\mathrm{E}(\Phi,A)$ has property $(\mathrm{TT})/\mathrm{T}$.
\end{theorem}

Because the proof of Theorem~\ref{theorem=TT/T} is involved, we first verify this for the case where $\Phi=A_{n-1}$ (hence $\mathrm{E}(\Phi,A)=\mathrm{E}(n,A)$) for $n\geq 3$; we then indicate the ingredients needed for the proof in the full generality.

\begin{proof}[Proof of Theorem~$\ref{theorem=TT/T}$ for $\Phi=A_{n-1}$ $(n\geq 3)$]

Set $\Gamma=\Gamma_0=\mathrm{E}(n,A)$ and $X$ is the set of all elementary root unipotents, namely, in this case,
\[
X:=\{e_{i,j}^a:1\leq i\leq n,1\leq j\leq n, i\ne j ,a\in A\}.
\]
Furthermore, set $Z$ as the union of all $\Gamma$-conjugates of $X$, that means,
\[
Z:= \bigcup_{h\in \Gamma} h^{-1}X h \subseteq \Gamma.
\]
We claim that this triple $(\Gamma,\Gamma_0,Z)$ fulfills all of the assumptions in the final part of Lemma~\ref{lemma=Shalom}. First, items $(i)$ and $(ii)$ are by construction. Property $(\mathrm{T})$ for $\Gamma$ is a big result, and was proved by the combination of  Shalom \cite{Shalom2006} and Vaserstein \cite{Vaserstein}; see also \cite{ErshovJaikinKassabov}. Therefore, the final process is to check that $\Gamma \supseteq Z$ has relative property $(\mathrm{TT})$. For this, we will show the following two assertions:
\begin{enumerate}[$(a)$]
  \item \underline{The pair $\Gamma \supseteq X$ has relative property $(\mathrm{TT})$.}
  \item \underline{The subset $Z$ is boundedly generated by $X$.}
\end{enumerate}
By a triangle inequality up to uniformly additive bounded error for quasi-cocycles, these two statements will immediately lead us to the conclusion. In what follows, we will verify, respectively, assertions $(a)$ and $(b)$.

\bigskip

\noindent
\underline{Proof of $(a)$:} Shalom \cite[Theorem~3.4]{Shalom1999} shows that the pair $\mathrm{E}(n-1,A)\ltimes A^{n-1} \trianglerighteq A^{n-1}$ has relative property $(\mathrm{T})$, where $\mathrm{E}(n-1,A)$ acts on $A^{n-1}$ by the natural left multiplication. Here we do not recall the definition of \textit{relative property $(\mathrm{T})$}; instead, we refer the reader to \cite[Definition~1.4.3]{BekkadelaHarpeValette} (there, this property is stated as \textit{property $(\mathrm{T})$ for a pair}). Then, we apply to a result by Ozawa \cite[Proposition~3: $(1)$ $\Rightarrow$ $(2)$]{Ozawa}, which states that for a pair of the form $G \ltimes H \trianglerighteq H$ for an \textit{abelian} $H$, relative property $(\mathrm{T})$ implies relative property $(\mathrm{TTT})$ of Ozawa. Again, we do not state the definition of (relative) property $(\mathrm{TTT})$ here, but we remark that relative property $(\mathrm{TTT})$ implies relative property $(\mathrm{TT})$ as in Definition~\ref{definition=relTT}; see \cite{Ozawa} for details. By combining these two results, we conclude that the pair $\mathrm{E}(n-1,A)\ltimes A^{n-1} \trianglerighteq A^{n-1}$ has relative property $(\mathrm{TT})$. Finally, embed $\mathrm{E}(n-1,A)\ltimes A^{n-1}$ into $\mathrm{E}(n,A)$ in such several (finitely many) ways that every $h \in X$ lies in the image of $A^{n-1}$ for some of such embeddings. This ends the proof of $(a)$.

\bigskip

\noindent
\underline{Proof of $(b)$:} We employ arguments in the work of Park and Woodburn \cite{ParkWoodburn} as follows (here we only sketch them). Take $s=e_{i,j}^a\in X$. Then for $\gamma\in \Gamma$, 
$$
\gamma s\gamma^{-1}= I_n+(i\textrm{-th column vector of }\gamma)\cdot a\cdot (j\textrm{-th column vector of }\gamma^{-1}).
$$
Let $\mathbf{v}={}^t(v_1,\ldots ,v_n)$ be the $i$-th column vector of $\gamma$, $\mathbf{w}=(w_1,\ldots ,w_n)$ be $a$ times the $j$-th row vector of $\gamma^{-1}$, and $(\gamma_1,\ldots ,\gamma_n)$ be the $i$-th row vector of $\gamma^{-1}$. Then since $i\ne j$, 
$
\sum_{l=1}^{n}\gamma_lv_l=1$ and $\sum_{l=1}^{n}w_lv_l=0$. 
Therefore, by letting $b_{l,m}=w_l\gamma_m -w_m \gamma_l$, we have that 
$
\mathbf{w}=\sum_{(1\leq)l<m(\leq n)}b_{l,m}(v_m\mathbf{e}_l-v_l\mathbf{e}_m)
$, 
and hence (the following are due to Suslin),
\begin{align*}
 \gamma s\gamma^{-1}=&I_n + \mathbf{v}\cdot \mathbf{w} = I_n + \mathbf{v}\cdot \Big( \sum_{(1\leq)l<m(\leq n)}b_{l,m}(v_m\mathbf{e}_l-v_l\mathbf{e}_m) \Big) \\
 =& I_n + \sum_{(1\leq)l<m(\leq n)} \mathbf{v}\cdot b_{l,m}(v_m\mathbf{e}_l-v_l\mathbf{e}_m) \\
 =& \prod_{(1\leq)l<m(\leq n)} (I_n + \mathbf{v}\cdot b_{l,m}(v_m\mathbf{e}_l-v_l\mathbf{e}_m)).
\end{align*}
Mennicke observed that each factor in the product on the very below side of the equalities above can be written as a product of bounded numbers of (only depend on $n$ and does not depend on $\gamma$ and $s$), precisely at most $8+2(n-2)=2n+4$, elementary matrices. Therefore, for every $\gamma\in \Gamma$ and for every $s\in X$, $\gamma s\gamma^{-1}$ is a product of at most $(n+2)n(n-1)$ elements in $X$; this bound is independent of the choices of $\gamma$ and $s$. For more details, consult Lemma 2.3, Lemma 2.6, and Corollary 2.7 in \cite{ParkWoodburn}.

Therefore, we have completed the proof of Theorem~\ref{theorem=TT/T} for $\Gamma=\mathrm{E}(n,A)$.
\end{proof}

Note that, in comparison with Vaserstein's bounded generation \cite{Vaserstein} (see the end of Section~\ref{section=TTwm}), the bounded generation in assertion $(b)$ above is weaker, but that the proof is considerably less involved (for instance, the proof of $(b)$ does not even need the finite generation of $A$: only the commutativity of $A$ is needed). From this point of view, it may be natural to expect that similar ``weak bounded generations" remain true for other root systems $\Phi$. As we will see below, this is \textit{indeed the case}: for instance, if $\Phi=C_n$ (symplectic case), then \cite[Section~1]{Kopeiko} implies it. In the full generality, we have the following, which may be seen as a new bounded generation ingredient to study fixed point properties and boundedness properties.

\begin{lemma}[New bounded generation observation]\label{lemma=boundedgeneration}
Let $A$ and $\Phi$ be as in Theorem~$\ref{theorem=main}$. Set $X:=X(\Phi,A)$, which denotes the set of all elementary root unipotents, and 
\[
Z:= \bigcup_{h\in \mathrm{E}(\Phi,A)} h^{-1}X h \subseteq \mathrm{E}(\Phi,A).
\]
Then, $Z$ is boundedly generated by $X$.
\end{lemma}

\begin{proof}
For the cases of $\Phi=A_{n-1}$, and $A=C_n$, we have referred to the relevant works, respectively, \cite{ParkWoodburn}, and \cite{Kopeiko}. In general case, a recent result of Stepanov \cite[Corollary~9.2]{Stepanov} yields the conclusion. In fact, Stepanov pointed out to the author that to prove the lemma above, we do not need to make full power of his result. It suffices to combine the normality of $\mathrm{E}(\Phi,A)$ in the Chevalley group, which is well-known and was established by Taddei, with the argument in the argument in \cite[proof of Theorem~9.1]{Stepanov}.
\end{proof}

\begin{proof}[Proof of Theorem~$\ref{theorem=TT/T}$ in general cases]

Set $\Gamma=\Gamma_0=\mathrm{E}(\Phi,A)$. Set $X$ and $Z \subseteq \Gamma$ as in Lemma~\ref{lemma=boundedgeneration}. We claim that this triple $(\Gamma,\Gamma_0,Z)$ fulfills all of the assumptions in Lemma~\ref{lemma=Shalom}. In this general case, property $(\mathrm{T})$ for $\Gamma$ is a celebrated result by Ershov, Jaikin-Zapirain and Kassabov \cite[Theorem~1.1]{ErshovJaikinKassabov}. Conditions $(i)$ and $(ii)$ obviously hold.

By virtue of Lemma~\ref{lemma=boundedgeneration}, what remains  is only to show that $\Gamma \supseteq X$ has relative property $(\mathrm{TT})$ (corresponding to assertion $(a)$ for $\Phi=A_{n-1}$). This can be done in the following three steps (compare with the proof for $\Phi=A_{n-1}$):
\begin{itemize}
  \item Show relative property $(\mathrm{T})$ for a suitable pair $G\ltimes H \trianglerighteq H$ for an abelian $H$,
  \item apply Ozawa's result \cite[Proposition~3]{Ozawa} and deduce that the pair above has relative property $(\mathrm{TT})$, and
  \item embed $G\ltimes H$ into $\Gamma$ in (several) appropriate ways.
\end{itemize}
More precisely, the combination of \cite[Theorem~7.10 and Corollary~7.11]{ErshovJaikinKassabov} (see also arguments  above Theorem~7.12 there and in the proof of it) will be employed in the first step. Note that in \cite{ErshovJaikinKassabov}, results are stated in terms of Steinberg-type groups, but that $\mathrm{E}(\Phi,A)$'s are group quotients of them. Here, observe, in addition, that relative property $(\mathrm{T})$ passes to group quotient pairs.

Thus we have established the theorem.
\end{proof}

\begin{remark}
As we mentioned above Theorem~\ref{theorem=TT/T}, we need to employ property $(\mathrm{T})$ for $\Gamma=\mathrm{E}(\Phi,A)$ for the proof of Theorem~\ref{theorem=TT/T}. For a general $\Phi$, this is a great achievement in \cite{ErshovJaikinKassabov} and the proof of property $(\mathrm{T})$ is highly involved. However, in a recent work \cite{Mimurancul}, the present author has obtained a much simpler proof of property $(\mathrm{T})$, which is applicable to many groups of such type, for instance, $\Phi=A_{n-1}$ and $\Phi=C_{n}$ for all $n\geq 3$. More precisely, we bypass ``$\epsilon$-orthogonality argument" in \cite{ErshovJaikinKassabov}, which requires delicate estimations of certain spectral quantities, by utilizing ``(intrinsic) \textit{upgrading without bounded generation}", which is invented in \cite{Mimurancul}. We refer the reader also to a short expository article \cite{Mimuraexpository}, which focuses on the proof of property $(\mathrm{T})$ for noncommutative universal lattices (see the next paragraph for the definition) along the line above. We mention that unlike the original proof in \cite{ErshovJaikinKassabov}, these less involved proofs in \cite{Mimurancul} and \cite{Mimuraexpository} provide no estimate of Kazhdan constants.

Another remark is on \textit{noncommutative universal lattices}. When $\Phi=A_{n-1}$, it is possible to drop the commutativity assumption on $A$. The group $\mathrm{E}(n,\mathbb{Z}\langle x_1,\ldots ,x_k\rangle)$ with $n\geq 3$ and $k$ finite is called the noncommutative universal lattice. Here $\langle\cdots\rangle$ denotes the noncommutative polynomial ring. Osin and the author asked whether some of noncommutative universal lattices is acylindrically hyperbolic. This question might remain open. The gap to apply Theorem~\ref{theorem=TT/T} is  the lack of the weak bounded generation, namely, there is \textit{no} reason to believe that Lemma~\ref{lemma=boundedgeneration} remains true for a noncommutative $A$. The current status of the question of whether some noncommutative universal lattice has property $(\mathrm{TT})/\mathrm{T}$ seems to be open. In the same work \cite{Mimurancul} as above, on the other hand, we show that noncommutative universal lattices, and even the corresponding Steinberg groups of type $A_{n-1}$, with $n\geq 4$, have the fixed point property with respect to $L_p$-spaces for \textit{all} $p\in (1,\infty)$.
\end{remark}

\begin{proof}[Proof of Theorem~$\ref{theorem=TTwm}$]
Combine Theorem~\ref{theorem=TT/T} with Proposition~\ref{proposition=ME}.
\end{proof}

\section{comparison with acylindrical hyperbolicity, and the proof of the main theorem}\label{section=ah}
In contrast to groups with property $(\mathrm{TT})_{\mathrm{wm}}$, Hamenst\"{a}dt \cite{Hamenstadt} showed that every non-elementary subgroup $H$ of $\mathrm{MCG}(\Sigma_g)$ admits plenty of unbounded quasi-cocycles into $(\lambda_H, \ell_2(H))$ (the left-regular representation). It turns out that her theorem applies to all acylindrically hyperbolic groups; see \cite[Theorem~8.3]{Osin}.

Before proceeding to the proofs, we recall that property $(\mathrm{TT})_{\mathrm{wm}}$ implies property $(\mathrm{T})$ (Lemma~\ref{lemma=TTwmT}), and that every group homomorphism from a discrete group with property $(\mathrm{T})$ into a discrete amenable (such as virtually abelian) group has a finite image (see \cite[Corollary~1.3.5]{BekkadelaHarpeValette}). Here we say that a group is \textit{virtually $P$} if there exists a subgroup of finite index that satisfies the property $P$.

\begin{proof}[Proof of Proposition~$\ref{proposition=HullOsin}$]
By \cite[Theorem~1.1]{Osin}, the image $H$ is either elliptic, virtually $\mathbb{Z}$, or acylindrically hyperbolic. Because $\Lambda$ and hence $H$ have, in particular by $(\star)$, property $(\mathrm{T})$, the second option is impossible. Suppose, on the contrary, that $H$ is acylindrically hyperbolic. Then by the argument above (see also \cite[Corollary~1.5]{HullOsin} and the main theorem of \cite{BestvinaBrombergFujiwara}), there must exist unbounded quasi-cocycles into $(\lambda_H, \ell_2(H))$. Note that because such an $H$ is infinite, $\lambda_H$ is weakly mixing (in fact, strongly mixing). This contradicts property $(\mathrm{TT})_{\mathrm{wm}}$ for $H$, which is deduced from Proposition~\ref{proposition=ME}.$(1)$. 

Therefore, $H\leqslant G$ must be absolutely elliptic.
\end{proof}

\begin{proof}[Proof of Theorem~$\ref{theorem=main}$]
Item $(i)$ immediately follows from Theorem~\ref{theorem=TTwm} and Proposition~\ref{proposition=HullOsin}. To show item $(ii)$, we employ subgroup classifications of $\mathrm{MCG}(\Sigma_g)$ and of $\mathrm{Out}(F_N)$. Theorem 8.10 in \cite{DahmaniGuirardelOsin} shows that if $H\leqslant \mathrm{MCG}(\Sigma_g)$ is not virtually abelian, then there exists a finite index subgroup $H_0\leqslant H$ that maps onto an acylindrically hyperbolic group. Hence, by item $(i)$, which we have verified, and Proposition~\ref{proposition=ME}, the image of $\Lambda$ must be virtually abelian. Again, property $(\mathrm{T})$ implies that the image must be finite. In the $\mathrm{Out}(F_N)$-target case, we need more care. However, the argument in the proof of \cite[Proposition~2.1]{BridsonWade} remains to work with no essential changes. Indeed, the absolute ellipticity excludes all fully irreducible automorphism classes; property $(\mathrm{T})$ then may take place of the $\mathbb{Z}$-aversion. Finally, \cite[Corollary~2.9 and Proposition~2.1]{BridsonWade} imply that $\phi(\Lambda)\cap \overline{\mathrm{IA}}_n$ and $\phi(\Lambda)$ in their statements are finite even in our current setting. It completes our proof.
\end{proof}

\begin{remark}\label{remark=quasi-hom}
There have been several results by several researchers on \textit{quasi-homomorphisms} (namely, quasi-cocycles into the trivial representation $(1_{\Gamma},\mathbb{R})$, also known as \textit{quasimorphisms}) on acylindrically hyperbolic groups and on higer rank lattices before corresponding results for quasi-cocycles into nontrivial unitary representations. The reason why we employ (twisted) quasi-cocycles rather than quasi-homomorphisms is the following: it might be open whether there exist unbounded quasi-homomorphisms on $\mathrm{E}(\Phi,A)$, even for $\Phi=A_{n-1}$ and $A=\mathbb{Z}[x]$. A partial result is given by the author \cite{Mimuraqh}: more precisely, if $n\geq 6$ and if $A$ is a euclidean domain, then all quasi-homomorphisms on $\mathrm{E}(n,A)(=\mathrm{SL}(n,A)$, viewed as a possibly uncountable discrete group) are bounded.
\end{remark}

\begin{remark}
After the first draft of this paper was ready, Osin pointed out to the author that, in order to show only that $\mathrm{E}(\Phi, A)$ in Theorem~\ref{theorem=main} itself is not acylindrically hyperbolic, we do not need to appeal to property $(\mathrm{T})$ for that group. More precisely, he provided the author with the following lemma. The author is grateful to him for letting the author include this result in the present paper.
\begin{lemma}[Osin]\label{lemma=Osin}
Let $\Gamma$ be a group, $\Gamma_0$ a subgroup of $\Gamma$, $Z$ a subset of $\Gamma$. Assume that $(\Gamma,\Gamma_0,Z)$ satisfies $(i')$: the subgroup $\langle Z\rangle \leqslant \Gamma$ generated by $Z$ is $\mathrm{not}$ virtually cyclic, and conditions $(ii)$ and $(iii)$ as in Lemma~$\ref{lemma=Shalom}$. If, besides, $\Gamma \supseteq Z$ has relative property $(\mathrm{TT})$, then $\Gamma$ is not acylindrically hyperbolic.

In particular, if $\Gamma_0=\Gamma$, then assumption $(iii)$ is automatically fulfilled.
\end{lemma}

We exhibit a sketch of Osin's proof. For the terminologies below, see \cite{Osin}. 
\begin{proof}
By the way of contradiction, assume that $\Gamma$ acts acylindrically and non-elementarily by isometries on a hyperbolic geodesic space $S$. By \cite[Theorem 1.1]{Osin}, there are 3 cases: with respect to the action on $S$, $\Gamma_0$ is non-elementary, $\Gamma_0$ is elliptic, or $\Gamma_0$ is virtually cyclic and contains a loxodromic element $h$. 

We deal with the first case. By \cite[Lemma~6.5 and Theorem~6.14]{DahmaniGuirardelOsin}, we can find a loxodromic element $h\in \Gamma_0$ such that $E(h)=\langle h\rangle \times K(\Gamma)$. Here $E(h)$ is the maximal elementary (in other words, virtually cyclic) subgroup of $\Gamma$ containing $h$, and $K(\Gamma)$ is the maximal finite normal subgroup of $\Gamma$. By $(i')$, there exists $z\in Z$ such that $z\notin E(h)$. Let $q$ be the quasi-cocycle from $\Gamma$ into $(\lambda_{\Gamma /E(h)},\ell_2(\Gamma/E(h)))$, provided by \cite{HullOsin}, which extends the natural homomorphism $E(h)\to E(h)/K(\Gamma)\cong \mathbb{Z}$. Then, from the construction of $q$ in \cite{HullOsin} (see the formula above Lemma~4.7 therein), $\|q (h^{-n} zh^n)\| \to \infty $ as $n\to \infty$. This contradicts $(ii)$ and relative property $(\mathrm{TT})$ in the assumption.

Then, we discuss the second and the third cases. In the second case, by the construction of quasi-cocycles in \cite{HullOsin}, we can easily obtain an unbounded quasi-cocycle $q$ from $\Gamma$ into $(\lambda_{\Gamma},\ell_2(\Gamma))$ that is bounded on $\Gamma_0$. This contradicts $(iii)$ and relative property $(\mathrm{TT})$. The third case may be reduced to the second one. Indeed, in that case, $E(h)\hookrightarrow _h  \Gamma$ by \cite[Theorem~6.8]{DahmaniGuirardelOsin}. Now by  \cite[Theorem 5.4]{Osin}, we can construct a non-elementary acylindrical action of $\Gamma$ by isometries on another hyperbolic geodesic space such that $E(h)$ (and hence $\Gamma_0$) is elliptic with respect to this action.
\end{proof}

From Lemma~\ref{lemma=Osin}, if we know that all of the  finite index subgroups of $\Gamma$ have finite abelianization, then we have the same conclusion as items $(i)$ and $(ii)$ of Theorem~\ref{theorem=main} for $\Lambda=\Gamma=\mathrm{E}(\Phi,A)$, without appealing to property $(\mathrm{T})$ (the condition above on abelianization is needed to exclude the case that the image is virtually $\mathbb{Z}$). To have the full result in Theorem~\ref{theorem=main}, not only for homomorphisms from $\Gamma$ but also for ones from a general $\Lambda$, we may need to employ property $(\mathrm{TT})/\mathrm{T}$ and property $(\mathrm{TT})_{\mathrm{wm}}$, as we have seen in the present paper. This is similar to the original situation in the Farb--Kaimanovich--Masur and the Bridson--Wade superrigidity: in their results, the case of non-uniform lattices is easier, and that the case of uniform lattices is harder.
\end{remark}

\section{Final remark on (in)finite generation of rings}\label{section=notfg}
As we mentioned in Remark~\ref{remark=notfg}, we prove the following.

\begin{proposition}\label{proposition=notfg}
The assertions of items $(i)$ and $(ii)$ in Theorem~$\ref{theorem=main}$ for $\Lambda=\Gamma$ $($a quotient of a finite index subgroup of $\mathrm{E}(\Phi,A)$$)$ remain true even for a unital, commutative, and associative ring $A$, possibly $\mathrm{not}$ finitely generated.
\end{proposition}

Here $\Gamma$ is viewed as a (possibly uncountable) discrete group.

\begin{proof}
It suffices to show the statement for the case that $\Gamma$ is a finite index subgroup of $\mathrm{E}(\Phi,A)$. Let $\phi \colon \Gamma \to G$ be a group homomorphism into an acylindrically hyperbolic group $G$. 

Suppose, on the contrary, that the image of $\Gamma$ is not absolutely elliptic. Then there exists $\gamma\in \Gamma$ such that $\phi(\gamma)$ acts as a loxodromic element with respect to some acylindrical $G$-action by isometries on a hyperbolic geodesic space. Write $\gamma$ as $\gamma=\gamma_1\cdots \gamma_m$, where $\gamma_i$, $1\leq i\leq m$, belongs to $X(\Phi,A)$ (the set of all elementary root unipotents). Set a \textit{finitely generated} subring $A_0$ of $A$ as the ring generated by $1\in A$ and all ring elements appearing in (coefficients of) $\gamma_1,\ldots ,\gamma_m$. Then, $(\gamma \in )\Gamma':=\Gamma \cap \mathrm{E}(\Phi,A_0)$ is of finite index in $\mathrm{E}(\Phi,A_0)$, and hence item $(i)$ of Theorem~\ref{theorem=main} applies to $\Gamma'$. It follows that $(\phi(\gamma) \in )\phi(\Gamma')$ must be absolutely elliptic in $G$, but it is a contradiction. This argument proves item $(i)$. Item $(ii)$ can be verified in a similar manner.
\end{proof}

\section*{acknowledgments}
The author wishes to express his gratitude to Alexei Stepanov and Nikolai Vavilov for answering his question on a weak bounded generation on Chevalley groups (Lemma~\ref{lemma=boundedgeneration}), to Marc Burger and Alessandra Iozzi for letting him know their forthcoming work, and to Denis Osin for Lemma~\ref{lemma=Osin}. He is grateful to Nicolas Monod for suggesting the symbol ``$(\mathrm{TT})/\mathrm{T}$" to the author. He thanks Martin R. Bridson, Cornelia Dru\c{t}u, Ursula Hamenst\"{a}dt, Andrei Jaikin-Zapirain, Yoshikata Kida, Shin-ichi Oguni, Andrei S. Rapinchuk, Igor Rapinchuk, and Richard D. Wade  for helpful comments. The referees gave suggestive comments, which considerably improved the readability of the present paper.

\bibliographystyle{amsalpha}
\bibliography{TTwm_fv.bib}

\providecommand{\bysame}{\leavevmode\hbox to3em{\hrulefill}\thinspace}
\providecommand{\MR}{\relax\ifhmode\unskip\space\fi MR }
\providecommand{\MRhref}[2]{%
  \href{http://www.ams.org/mathscinet-getitem?mr=#1}{#2}
}
\providecommand{\href}[2]{#2}
\begin{thebibliography}{BFGM07}

\bibitem[BBF]{BestvinaBrombergFujiwara}
M.~Bestvina, K.~Bromberg, and K.~Fujiwara, \emph{Bounded cohomology with
  coefficients in uniformly convex {B}anach spaces}, to appear in Comm. Math.
  Helv., arXiv:1306.1542v2.

\bibitem[BdlHV08]{BekkadelaHarpeValette}
B.~Bekka, P.~de~la Harpe, and A.~Valette, \emph{Kazhdan's property $({T})$},
  Combridge University Press, 2008.

\bibitem[BF02]{BestvinaFujiwara}
M.~Bestvina and K.~Fujiwara, \emph{Bounded cohomology of subgroups of mapping
  class groups}, Geom. Topl. \textbf{6} (2002), 69--89.

\bibitem[BFGM07]{BaderFurmanGelanderMonod}
U.~Bader, A.~Furman, T.~Gelander, and N.~Monod, \emph{Property $({T})$ and
  rigidity for actions on {B}anach spaces}, Acta Math. \textbf{198} (2007),
  no.~1, 57--105.

\bibitem[BI]{BurgerIozzi}
M.~Burger and A.~Iozzi, \emph{$\ell^2$-stability and homomorphism into mapping
  class groups}, forthcoming work.

\bibitem[BW11]{BridsonWade}
M.~R. Bridson and R.~D. Wade, \emph{Actions of higher-rank lattices on free
  groups}, Compositio Math. \textbf{147} (2011), 1573--1580.

\bibitem[DGO17]{DahmaniGuirardelOsin}
F.~Dahmani, V.~Guirardel, and D.~Osin, \emph{Hyperbolically embedded subgroups
  and rotating families in groups acting on hyperbolic spaces}, Mem. Amer.
  Math. Soc. \textbf{245} (2017), no.~1156, v+152. \MR{3589159}

\bibitem[EJZK]{ErshovJaikinKassabov}
M.~Ershov, A.~Jaikin-Zapirain, and M.~Kassabov, \emph{Property $({T})$ for
  groups graded by root system}, To appear in Memoirs of the Amer. Math. Soc.,
  arXiv:1102.0031v2.

\bibitem[FM98]{FarbMasur}
B.~Farb and H.~Masur, \emph{Superrigidity and mapping class groups}, Topology
  \textbf{37} (1998), 1169--1176.

\bibitem[Fur11]{Furmansurvey}
A.~Furman, \emph{A survey of measured group theory}, In Geometry, Rigidity, and
  Group Actions, The University of Chicago Press, Chicago and London, 2011.

\bibitem[GS14]{GruberSisto}
D.~Gruber and A.~Sisto, \emph{Infinitely presented graphical small cancellation
  groups are acylindrically hyperbolic.}, preprint, arXiv:1408.4488v3 (2014).

\bibitem[Ham08]{Hamenstadt}
U.~Hamenst\"{a}dt, \emph{Bounded cohomology and isometry groups of hyperbolic
  spaces}, J.\ Eur. Math. Soc. \textbf{10} (2008), 315--349.

\bibitem[HO13]{HullOsin}
M.~Hull and D.~Osin, \emph{Induced quasi-cocycles on groups with hyperbolically
  embedded subrgoups}, Alg. Geom. Topol. \textbf{13} (2013), 2635--2665.

\bibitem[KM96]{KaimanovichMasur}
A.~V. Kaimanovich and H.~Masur, \emph{The {P}oisson boundary of the mapping
  class group}, Invent. Math. \textbf{125} (1996), 221--264.

\bibitem[Kop78]{Kopeiko}
V.~I. Kope\u{\i}ko, \emph{Stabilization of symplectic groups over a ring of
  polynomials}, Math. USSR-Sb \textbf{34} (1978), 655--669.

\bibitem[Mim]{Mimuraup}
M.~Mimura, \emph{Property $({TT})$ modulo ${T}$ and homomorphism superrigidity
  into mapping class groups}, unpublished manuscript, arXiv:1106.3769.

\bibitem[Mim10]{Mimuraqh}
\bysame, \emph{On quasi-homomorphisms and commutators in the special linear
  group over a {E}uclidean ring}, Int. Math. Res. Notices \textbf{18} (2010),
  3519--3529.

\bibitem[Mim11]{Mimura1}
\bysame, \emph{Fixed point properties and second bounded cohomology of
  universal lattices on {B}anach spaces}, J. reine angew. Math. \textbf{653}
  (2011), 115--134.

\bibitem[Mim15]{Mimurancul}
\bysame, \emph{Upgrading fixed points without bounded generation}, preprint,
  forthcoming version (v3) of arXiv:1505.06728 (2015).

\bibitem[Mim16]{Mimuraexpository}
\bysame, \emph{An alternative proof of {K}azhdan property for elementary
  groups}, expository article, forthcoming version (v2) of arXiv:1611.00337
  (2016).

\bibitem[MO15]{MinasyanOsin}
A.~Minasyan and D.~Osin, \emph{Acylindrical hyperbolicity of groups acting on
  trees}, Math. Ann. \textbf{362} (2015), no.~3-4, 1055--1105.

\bibitem[Mon01]{Monod}
N.~Monod, \emph{Continuous bounded cohomology of locally compact groups},
  Springer Lecturenotes in Mathematics, vol. 1758, 2001.

\bibitem[Mon06]{MonodICM}
\bysame, \emph{An invitation to bounded cohomology.}, in International Congress
  of Mathematicians, vol.~II, Eur.\ {M}ath.\ {S}oc., Z\"{u}rich, 2006.

\bibitem[MS06]{MonodShalom}
N.~Monod and Y.~Shalom, \emph{Orbit equivalence rigidity and bounded
  cohomology}, Ann. of Math. (2) \textbf{164} (2006), 825--878.

\bibitem[Osa14]{Osajda}
D.~Osajda, \emph{Small cancellation labellings of some infinite graphs and
  applications}, preprint, arXiv:1406.5015v2 (2014).

\bibitem[Osa17]{Osajda2017}
\bysame, \emph{Residually finite non-exact groups}, preprint,
  arXiv:1703.03791v2 (2017).

\bibitem[Osi16]{Osin}
D.~Osin, \emph{Acylindrically hyperbolic groups}, Trans. Amer. Math. Soc.
  \textbf{368} (2016), no.~2, 851--888. \MR{3430352}

\bibitem[Oza11]{Ozawa}
N.~Ozawa, \emph{Quasi-homomorphism rigidity with noncommutative targets}, J.\
  reine angew. Math. \textbf{655} (2011), 89--104.

\bibitem[PW95]{ParkWoodburn}
H.~Park and C.~Woodburn, \emph{An algorithmic proof of {S}uslin's stability
  theorem for polynomial rings}, J.\ Algebra \textbf{178} (1995), 277--298.

\bibitem[Sha99]{Shalom1999}
Y.~Shalom, \emph{Bounded generation and {K}azhdan's property $({T})$}, Inst.
  Hautes \'{E}tudes Sci. Publ. Math. \textbf{90} (1999), 145--168.

\bibitem[Sha06]{Shalom2006}
\bysame, \emph{The algebraization of {K}azhdan's property $({T})$}, in
  International Congress of Mathematicians, vol. II, Eur. Math. Soc.,
  Z\"{u}rich (2006).

\bibitem[Ste68]{Steinberg}
R.~Steinberg, \emph{Lectures on {C}hevalley groups}, Yale University, 1968.

\bibitem[Ste16]{Stepanov}
A.~Stepanov, \emph{Structure of {C}hevalley groups over rings via universal
  localization}, J. Algebra \textbf{450} (2016), 522--548. \MR{3449702}

\bibitem[Vas07]{Vaserstein}
L.~Vaserstein, \emph{Bounded reduction of invertible matrices over polynomial
  ring by addition operators}, preprint (2007).

\end{thebibliography}

\end{document}